\documentclass{amsart}
\usepackage{amsfonts}
\usepackage{color}
\def\R{{\mathbb {R}}}
\def\N{{\mathbb {N}}}

\def\A{{\mathcal{A}}}

\def\E{{\mathcal{E}}}

\def\F{{\mathcal{F}}}

\def\ve{\varepsilon}

\newtheorem{teo}{Theorem}[section]
\newtheorem{lema}[teo]{Lemma}

\theoremstyle{remark}
\newtheorem{remark}[teo]{Remark}

\theoremstyle{definition}

\numberwithin{equation}{section}

\title[Uniqueness of minimal energy solutions]{Uniqueness of minimal energy solutions for a semilinear problem involving the fractional laplacian}
\author[J. Fern\'andez Bonder, A. Silva and J. Spedaletti]
{Juli\'an Fern\'andez Bonder, Anal\'{\i}a Silva and Juan Spedaletti}

\address{Juli\'an Fern\'andez Bonder \hfill\break\indent
Departamento  de Matem\'atica, FCEyN, Universidad de Buenos Aires,\hfill\break\indent
Instituto de Matem\'atica Santal\'o (IMAS), CONICET 
\hfill\break\indent Pabell\'on I, Ciudad Universitaria (1428),
Buenos Aires, Argentina.}

\email{{\tt jfbonder@dm.uba.ar}\hfill\break\indent {\it Web page:}
{\tt http://mate.dm.uba.ar/$\sim$jfbonder}}

\address{Anal\'{\i}a Silva \hfill\break\indent
Departamento de Matem\'atica, FCFMyN, Universidad Nacional de San
Luis \hfill\break\indent Instituto de Matem\'atica Aplicada San
Luis, IMASL, CONICET. \hfill\break\indent Italia avenue 1556, office
5, San Luis (5700), San Luis, Argentina.}

\email{{\tt acsilva@unsl.edu.ar}}

\address{Juan Spedaletti \hfill\break\indent
Departamento de Matem\'atica, FCFMyN, Universidad Nacional de San Luis \hfill\break\indent Instituto de Matem\'atica Aplicada San Luis, IMASL, CONICET.
\hfill\break\indent Italia avenue 1556, office 5, San Luis (5700),
San Luis, Argentina.}

\email{{\tt jfspedaletti@unsl.edu.ar}}

	

\begin{document}

\begin{abstract}
In this paper we study a semilinear problem for the fractional laplacian that are the counterpart of the Neumann problems in the classical setting. We show uniqueness of minimal energy solutions for small domains.
\end{abstract}

\maketitle

\section{Introduction}
In recent years there has been an increasing amount of attention to problems involving nonlocal diffusion operators. These problems are so vast that is impossible to give a comprehensive list of references, just to cite a few, see  \cite{DGLZ, Eringen, Giacomin-Lebowitz, Laskin, Metzler-Klafter, Zhou-Du} for some physical models, \cite{Akgiray-Booth, Levendorski, Schoutens} for some applications in finances, \cite{Constantin} for appications in fluid dynamics, \cite{Humphries, Massaccesi-Valdinoci, Reynolds-Rhodes} for application in ecology and \cite{Gilboa-Osher} for some applications in image processing.

Among these applications, one operator that is of particular importance is the {\em fractional laplacian} that is defined (up to some normalization constant) as
$$
(-\Delta)^s u(x) = \text{p.v.} \int_{\Omega} \frac{u(x)-u(y)}{|x-y|^{n+2s}}\, dy,
$$
where p.v. stands for {\em in principal value} and $\Omega\subset\R^n$ is a bounded domain.

This operator is classical and have been studied by several authors. See for instance \cite{Acosta-Borthagaray, BPS, Caffarelli-Silvestre, DelPezzo-Salort, DiNezza-Palatucci-Valdinoci, Demengel, LinLin, Sire-Vazquez-Volzone}, etc.

Recall that this operator is to the so called {\em regional} fractional laplacian that corresponds to Levy processes where the long jumps are restricted to be inside $\Omega$. See, for instance, \cite{Dipierro-RosOton-Valdinoci} for a discussion on this.

This operator is commonly used as a fractional version of the Neumann laplacian. In the literature, there exists several forms of fractional version of the Neumann laplacian. We refer the interested reader again to the article \cite{Dipierro-RosOton-Valdinoci} for more on this.

In this paper, we address the following semilinear problem associated to $(-\Delta)^s$
\begin{equation}\label{1.1}
(-\Delta)^s u + u = \lambda |u|^{q-2}u,\quad \text{in } \Omega.
\end{equation}
This is the fractional counterpart of the classical Neumann problem
\begin{equation}\label{Neumann}
\begin{cases}
-\Delta u + u = \lambda |u|^{q-2}u & \text{in }\Omega\\
\frac{\partial u}{\partial\nu} = 0 & \text{on }\partial\Omega.
\end{cases}
\end{equation}
Here, $q$ is a {\em subcritical exponent} in the sense of the Sobolev embeddings. That is,
$$
1\le q< 2^*_s := \frac{2n}{n-2s}.
$$

The problem can be separated into three different types of behavior: sublinear ($1\le q<2$); linear ($q=2$) and superlinear ($q>2$).

The linear case is by now well understood as an eigenvalue problem and will not be considered here.

For the sublinear and superlinear cases, the parameter $\lambda$ is superfluous since if one gets a solution for some particular value of $\lambda$, then taking a suitable multiple of the solution the value of $\lambda$ can be taken to be 1.

It is fairly easy to see that any solution to \eqref{1.1} is a critical point of the functional
$$
\F(u) := \frac{\frac12 [u]_{s;\Omega}^2 + \|u\|_{2;\Omega}^2}{\|u\|_{q;\Omega}^2},
$$
where 
$$
[u]_{s;\Omega} := \left(\iint_{\Omega\times\Omega} \frac{(u(x)-u(y))^2}{|x-y|^{n+2s}}\, dxdy\right)^\frac12
$$
is the so-called {\em Gagliardo seminorm} of $u$ and, as usual, $\|u\|_{r;\Omega}$ denotes the $L^r(\Omega)-$norm.

By standard variational methods, one can see that there exists {\em minimal energy solutions} to \eqref{1.1}. That is, functions $u\in H^s(\Omega)$ such that
$$
\F(u) = \inf_{v\in H^s(\Omega)} \F(v).
$$
Moreover, by a direct application of the Ljusternik -- Schnirelmann method, one can construct a sequence $\lambda_k\uparrow \infty$ of critical energy levels and a sequence of critical points $\{u_k\}_{k\in\N}$ of $\F$ associated to $\{\lambda_k\}_{k\in\N}$.

Therefore, there exist infinitely many solutions to problem \eqref{1.1}.

\medskip

In this paper, we focus on {\em minimal energy solutions} to \eqref{1.1}. In particular to the multiplicity problem of such solutions.

To this end, inspired by the results of \cite{Bonder-LamiDozo-Rossi}, given a domain $\Omega$ we consider the family of contracted domains
\begin{equation}\label{Omega.mu}
\Omega_\ve := \ve\cdot\Omega = \{\ve x\colon x\in \Omega\} \quad \text{as }\ve\downarrow 0
\end{equation}
and look for the asymptotic behavior of minimal energy solutions in $\Omega_\ve$ as $\ve\downarrow 0$.

We first show that the asymptotic behavior of every minimal energy solution is the same and, using this asymptotic behavior, we are able to conclude the uniqueness of minimal energy solution for contracted domains.

Finally, we give an estimate on the contraction parameter in order to obtain the uniqueness result.

To end this introduction, we want to remark that the same ideas can be used to deal with the Neumann problem \eqref{Neumann}. The changes needed are easy and are left to the interested reader.

We want to recall that uniqueness results for problem \eqref{Neumann} has been considered in the literature before. See for instance \cite{Lin-Ni-Takagi}.

This result is a first step in an investigation which might be pursued in various directions: for instance, taking different Neumann fractional operators, as the one considered in \cite{Dipierro-RosOton-Valdinoci}, or show if the known uniqueness results for the classical Neumann laplacian hold in this situation. We leave these questions for further investigation.

\section{Preliminaries.}

Let $\Omega\subset \R^n$ be a bounded smooth domain and $0<s<1$. The fractional order Sobolev space is defined as
$$
H^s(\Omega) := \left\{u\in L^2(\Omega)\colon \frac{u(x)-u(y)}{|x-y|^{\frac{n}{2}+s}}\in L^2(\Omega\times\Omega)\right\}.
$$
This space is endowed with the norm
$$
\|u\|_{s; \Omega} := ([u]_{s;\Omega}^2 + \|u\|_{2;\Omega}^2)^\frac{1}{2}.
$$

It is well known (see, for instance, \cite{DiNezza-Palatucci-Valdinoci}) that there exists a critical exponent
$$
2^*_s := \begin{cases}
\frac{2n}{n-2s} & \text{if } 2s<n,\\
\infty & \text{otherwise}
\end{cases}
$$
such that for any $1\le q<2^*_s$ the embedding $H^s(\Omega)\subset L^q(\Omega)$ is compact.

We define the Sobolev constant as the number
$$
S(\Omega) = S_{s,q}(\Omega) := \inf_{u\in H^s(\Omega)} \frac{\|u\|_{s;\Omega}^2}{\|u\|_{q;\Omega}^2}.
$$

It is easy to see, as a consequence of the compactness of the embedding, that there exists an extremal for $S(\Omega)$. That is a function $u\in H^s(\Omega)$ where the above infimum is attained.

Also, any extremal for $S(\Omega)$ is a minimal energy (weak) solution to \eqref{1.1}.

The constant $\lambda$ in \eqref{1.1} depends on the normalization of the extremal. For instance, if the extremal $u$ is normalized as $\|u\|_{q;\Omega} = 1$ then $\lambda = S(\Omega)$.

Recall that the operator $(-\Delta)^s$ is a bounded operator between the Sobolev space $H^s(\Omega)$ and its dual $(H^s(\Omega))'$ and can be computed by
$$
\langle (-\Delta)^s u, v\rangle = \frac12 \iint_{\Omega\times\Omega} \frac{(u(x)-u(y))(v(x)-v(y))}{|x-y|^{n+2s}}\, dxdy.
$$

Therefore, for a solution to \eqref{1.1} we mean a function $u\in H^s(\Omega)$ such that
\begin{align*}
\frac12 \iint_{\Omega\times\Omega} \frac{(u(x)-u(y))(v(x)-v(y))}{|x-y|^{n+2s}}\, dxdy + \int_\Omega uv\, dx = \lambda \int_\Omega |u|^{q-2}uv\, dx,
\end{align*}
for every $v\in H^s(\Omega)$.

\section{Asymptotic behavior in thin domains}

Throughout this section we fix the exponent $q\in [1,2^*_s)$, $q\neq 2$.

Our objective in this section is to study the asymptotic behavior of the constant $S(\Omega_\ve)$ as $\ve\downarrow 0$, where the contracted domains $\Omega_\ve$ are given by \eqref{Omega.mu}. 

We begin with a simple estimate.
\begin{lema}\label{lema1}
Under the above notations, we have that
$$
S(\Omega_\ve) \le |\Omega_\ve|^{1-\frac{2}{q}} = \ve^{n(1-\frac{2}{q})} |\Omega|^{1-\frac{2}{q}}.
$$
\end{lema}

\begin{proof}
The lemma follows just by taking $u=1$ as a test function in the definition of $S(\Omega_\ve)$.
\end{proof}

Now we want to be more precise. We need to show that the asymptotic behavior of $S(\Omega_\ve)$ is given precisely by $\ve^{n(1-\frac{2}{q})}$ and also to study the behavior of the associated extremals.
\begin{lema}\label{lema2}
Under the above notations, we have that
$$
\lim_{\ve\downarrow 0} \frac{S(\Omega_\ve)}{\ve^{n(1-\frac{2}{q})}} = |\Omega|^{1-\frac{2}{q}}.
$$
Moreover if $u_\ve\in H^s(\Omega_\ve)$ is an extremal for $S(\Omega_\ve)$, the rescaled extremals $\bar u_\ve(x) := u_\ve(\ve x)$ normalized such that $\|\bar u_\ve\|_{q;\Omega} = 1$ verify that
$$
\bar u_\ve \to |\Omega|^{-\frac{1}{q}}\qquad \text{strongly in } H^s(\Omega). 
$$
\end{lema}

\begin{proof}
First, observe that for $v\in H^s(\Omega_\ve)$, if we denote $\bar v(x) = v(\ve x)$, then $\bar v\in H^s(\Omega)$. Moreover, $[v]_{s; \Omega_\ve} = \ve^{\frac{n}{2}-s} [\bar v]_{s;\Omega}$ and $\|v\|_{r;\Omega_\ve} = \ve^\frac{n}{r} \|\bar v\|_{r;\Omega}$ for $1\le r<2^*_s$. Therefore
$$
\frac{\|v\|_{s;\Omega_\ve}^2}{\|v\|_{q;\Omega_\ve}^2} = \ve^{n(1-\frac{2}{q})} \frac{\ve^{-2s}[\bar v]_{s;\Omega}^2 + \|\bar v\|_{2;\Omega}^2}{\|\bar v\|_{q;\Omega}^2}.
$$

Now, let $u_\ve\in H^s(\Omega_\ve)$ be an extremal for $S(\Omega_\ve)$ and let $\bar u_\ve(x) = u_\ve(\ve x)$. Then,
$$
S(\Omega_\ve) =  \ve^{n(1-\frac{2}{q})} \frac{\ve^{-2s}[\bar u_\ve]_{s;\Omega}^2 + \|\bar u_\ve\|_{2;\Omega}^2}{\|\bar u_\ve\|_{q;\Omega}^2}.
$$

Now, by Lemma \ref{lema1}, it follows that
\begin{equation}\label{estimate1}
\frac{\ve^{-2s}[\bar u_\ve]_{s;\Omega}^2 + \|\bar u_\ve\|_{2;\Omega}^2}{\|\bar u_\ve\|_{q;\Omega}^2}\le |\Omega|^{1-\frac{2}{q}}.
\end{equation}
Let us now fix the normalization of the extremal $u_\ve$ such that $\|\bar u_\ve\|_{q;\Omega}=1$, and by \eqref{estimate1}, we obtain that $\bar u_\ve$ is bounded in $H^s(\Omega)$ uniformly on $\ve>0$. So, there exists $\bar u\in H^s(\Omega)$ such that (up to some sequence $\ve_k\to 0$),
\begin{align}
\label{convWsp}&\bar u_\ve \rightharpoonup \bar u \qquad \text{weakly in } H^s(\Omega),\\
\label{convLr}&\bar u_\ve \to \bar u \qquad \text{strongly in } L^r(\Omega) \text{ for any } 1\le r < 2^*_s.
\end{align}
Also, from \eqref{estimate1} and \eqref{convWsp}, we have that
$$
[\bar u]_{s;\Omega}\le \liminf_{\ve\downarrow 0}\, [\bar u_\ve]_{s;\Omega} = 0, 
$$
therefore $\bar u$ is constant and, since $\|\bar u_\ve\|_{q;\Omega}=1$, from \eqref{convLr} we obtain that $\|\bar u\|_{q;\Omega}=1$.

All these together imply that $\bar u = |\Omega|^{-\frac{1}{q}}$.

From these estimates, one easily concludes that
\begin{align*}
|\Omega|^{1-\frac{2}{q}} &\le \liminf_{\ve\downarrow 0} \ve^{-2s}[\bar u_\ve]_{s;\Omega}^2 + \|\bar u_\ve\|_{2;\Omega}^2 \\
&= \liminf_{\ve\downarrow 0} \frac{S(\Omega_\ve)}{\ve^{n(1-\frac{2}{q})}}\le \limsup_{\ve\downarrow 0} \frac{S(\Omega_\ve)}{\ve^{n(1-\frac{2}{q})}} \le |\Omega|^{1-\frac{2}{q}}.
\end{align*}
The proof is complete.
\end{proof}

\section{Uniqueness of extremals for small domains}

In this section we show the uniqueness of extremals if the domain is contracted enough.

For that purpose, observe that if $u_\ve$ is an extremal for $S(\Omega_\ve)$ and $\bar u_\ve$ is the rescaled extremal normalized as $\|\bar u_\ve\|_{q;\Omega}=1$, then $\bar u_\ve$ is a weak solution of the problem
\begin{equation}\label{ecuacion.2}
(-\Delta)^s u + \ve^{2s} u = \ve^{2s} \lambda_\ve |u|^{q-2}u \qquad \text{in }\Omega,
\end{equation}
where $\lambda_\ve = S(\Omega_\ve) \ve^{-n(1-\frac{2}{q})}$. Recall also, that $0\le \lambda_\ve \le |\Omega|^{1-\frac{2}{q}}$ (c.f. Lemma \ref{lema1}).

So, we define the space
$$
\E := \{u\in H^s(\Omega)\colon \|u\|_{q;\Omega} = 1\}.
$$
It is easy to see that $\E$ is a $C^1$ manifold.

We then define $F\colon \E\times [0,1)\to (H^s(\Omega))'$ as
\begin{align*}
\langle F(u,\ve), v\rangle := &\frac12\iint_{\Omega\times\Omega} \frac{(u(x)-u(y))(v(x)-v(y))}{|x-y|^{n+2s}}\, dxdy\\
& + \ve^{2s}\int_\Omega uv\, dx - \ve^{2s}\lambda_\ve \int_\Omega |u|^{q-2}uv\, dx.
\end{align*}

Denote $u_0 = |\Omega|^{-\frac{1}{q}}\in \E$ and observe that $F(u_0,0)=0$.

Following the ideas of \cite{Bonder-LamiDozo-Rossi} we want to use the Implicit Function Theorem (IFT) to show the existence of a small number $\delta>0$ and a curve $\phi\colon [0,\delta)\to \E$ such that
$$
\phi(0)=u_0\quad \text{and}\quad F(\phi(\ve),\ve) = 0 \ \text{for every } 0\le \ve<\delta,
$$
and if $(u,\ve)\in \E\times [0,\delta)$ is such that $F(u,\ve)=0$ and $u$ is close to $u_0$ then $u=\phi(\ve)$.

Observe that if we can apply the IFT then, combining this with Lemma \ref{lema2}, automatically we obtain the uniqueness of extremals of $S(\Omega_\ve)$ for $\ve$ small.

In order to be able to apply the IFT we need to check that $d_u F|_{(u_0,0)}$ is invertible (see \cite{Ambrosetti-Prodi} or \cite{Upmeier}). Recall that since $F$ is define on a manifold, the derivative is defined on the tangent space of $\E$ at the point $u_0$.  

Let us begin with a couple of lemmas.
\begin{lema}\label{tangente}
The tangent space of $\E$ at $u_0$, that we denote by $T_{u_0}\E$ is given by
$$
T_{u_0}\E = \left\{v\in H^s(\Omega)\colon \int_\Omega v\, dx = 0\right\}.
$$
\end{lema}

\begin{proof}
Let $v\in T_{u_0}\E$. Then, there exists a differentiable curve, $\alpha\colon (-1,1)\to \E$ such that $\alpha(0)=u_0$ and $\dot\alpha(0)=v$.

But, since $\alpha(t)\in \E$ for every $t\in (-1,1)$ it follows that
$$
\int_\Omega |\alpha(t)|^q\, dx = 1 \quad \text{for every } t\in (-1,1).
$$

Differentiating both sides of the equality gives
$$
\int_\Omega q |\alpha(t)|^{q-2}\alpha(t)\dot\alpha(t)\, dx = 0.
$$
So, if we evaluate at $t=0$ and recall that $u_0$ is constant, we obtain that 
\begin{equation}\label{v=0}
\int_\Omega v\, dx = 0.
\end{equation}

On the other hand, if $v\in H^s(\Omega)$ verifies \eqref{v=0}, we construct the curve $\alpha\colon (-1,1)\to\E$ as
$$
\alpha(t) = \frac{u_0 + tv}{\|u_0+tv\|_q}.
$$
Straightforward computations show that $\alpha(0)=u_0$ and $\dot\alpha(0)=v$.
\end{proof}

Now, we denote $\A = (\text{span}\{1\})^\perp = \{f\in (H^s(\Omega))'\colon \langle f, 1\rangle = 0\}$.

\begin{lema}\label{rango}
We have that 
$$
d_u F|_{(u_0,0)}\colon T_{u_0}\E\to \A.
$$
Moreover, the following expression holds
$$
\langle d_u F|_{(u_0,0)}(u), v\rangle = \frac12 \iint_{\Omega\times\Omega} \frac{(u(x)-u(y))(v(x)-v(y))}{|x-y|^{n+2s}}\, dxdy.
$$
\end{lema}

\begin{proof}
To prove the Lemma, first observe that
$$
\langle F(u,0), v\rangle =  \frac12 \iint_{\Omega\times\Omega} \frac{(u(x)-u(y))(v(x)-v(y))}{|x-y|^{n+2s}}\, dxdy,
$$
for every $u\in \E$.

From this expression the lemma follows.
\end{proof}

It remains to see that $d_u F|_{(u_0,0)}$ has a continuous inverse.
\begin{lema}\label{inversa.continua}
The derivative $d_u F|_{(u_0,0)}\colon T_{u_0}\E\to \A$ has a continuous inverse.
\end{lema}

\begin{proof}
First observe that $T_{u_0}\E$ is a Hilbert space with inner product given by
$$
(u,v) = \langle d_u F|_{(u_0,0)}(u), v\rangle =  \frac12 \iint_{\Omega\times\Omega} \frac{(u(x)-u(y))(v(x)-v(y))}{|x-y|^{n+2s}}\, dxdy.
$$
So, the Lemma follows from the Riesz representation Theorem.
\end{proof}

Combining Lemmas \ref{tangente}, \ref{rango}, \ref{inversa.continua} we are in position to apply the IFT and conclude the main result of the section
\begin{teo}\label{existe.mu0}
Given $\Omega\subset\R^n$ smooth and of finite measure and $1< q <2^*_s$, there exists $\delta>0$ such that $S(\Omega_\ve)$ has a unique extremal for $0<\ve<\delta$.
\end{teo}

\begin{proof}
At this point the proof is a direct consequence of the IFT and the remarks made at the beginning of the section.
\end{proof}

\section{Estimates for the contraction parameter}

Let us first define 
$$
\ve_0 = \sup\{\delta>0\colon \exists ! \text{ normalized extremal for } S(\Omega_\ve) \ \forall \ve<\delta\}.
$$
From the results of the previous section, we know that $\ve_0>0$. We now want to find a lower bound for $\ve_0$.

\begin{lema}\label{lema.clave}
There exists $u_0\in H^s(\Omega_{\ve_0})$ an extremal for $S(\Omega_{\ve_0})$, such that the rescaled function $\bar u_0(x) = u_0(\ve_0 x)$ normalized as $\|\bar u_0\|_{q;\Omega}=1$ verifies that $d_u F|_{(\bar u_0,\ve_0)}$ is not invertible.
\end{lema}

\begin{proof}
Assume the oposite. 

We first claim that there is a unique extremal for $S(\Omega_{\ve_0})$. Otherwise, if $u_0\neq u_1$ are extremals such that the rescaled functions $\bar u_i(x) = u_i(\ve_0 x)$, $i=0,1$ normalized as $\|\bar u_i\|_{q;\Omega}=1$, $i=0,1$, verify that $d_u F|_{(\bar u_i,\ve_0)}$ is invertible for $i=0,1$. But then, by the IFT, there exists $\delta>0$ and two curves $\phi_i\colon (\ve_0-\delta, \ve_0+\delta)\to \E$ such that
$$
F(\phi_i(\ve),\ve)=0,\quad \text{for every }\ve\in (\ve_0-\delta, \ve_0+\delta),\ i=0,1.
$$
But this contradicts the uniqueness of extremals for $\ve<\ve_0$.

Now, let $\ve_k>\ve_0$ be such that $\ve_k\to\ve_0$ as $k\to\infty$ and let $u_k$ be an extremal for $S(\Omega_{\ve_k})$. We normalized these extremals so that the rescaled functions $\bar u_k(x) = u_k(\ve_k x)$ verify $\|\bar u_k\|_{q; \Omega}=1$.

We want to see that $\{\bar u_k\}_{k\in\N}$ converges to $\bar u_0$ that is the rescaled function of the unique extremal $u_0$ for $S(\Omega_{\ve_0})$.

In fact, it is immediate to see that $\sup_{k\in\N} \|\bar u_k\|_{s;\Omega}<\infty$, and so, up to a subsequence, there exists $\bar w\in H^s(\Omega)$ such that
\begin{align*}
& \bar u_k \rightharpoonup \bar w \quad \text{weakly in } H^s(\Omega)\\
& \bar u_k \to \bar w \quad \text{strongly in } L^r(\Omega), \text{ for every } 1\le r<2^*_s.
\end{align*}
From these convergence results it follows that
$$
1=\|\bar u_k\|_{q;\Omega}\to \|\bar w\|_{q;\Omega}
$$
and
$$
\ve_0^{-2s} [\bar w]_{s,2;\Omega}^2 + \|\bar w\|_{2;\Omega}^2 \le \liminf_{k\to\infty} \ve_k^{-2s} [\bar u_k]_{s,2;\Omega}^2 + \|\bar u_k\|_{2;\Omega}^2 =\liminf_{k\to\infty} \frac{S(\Omega_{\ve_k})}{\lambda_{\ve_k}}.
$$

Now, let $u_0\in H^s(\Omega_{\ve_0})$ be the unique extremal for $S(\Omega_{\ve_0})$ normalized such that the rescaled function $\bar u_0$ satisfies $\|\bar u_0\|_{q;\Omega}=1$.

Then,
$$
\limsup_{k\to\infty}\frac{S(\Omega_{\ve_k})}{\lambda_{\ve_k}} \le \limsup_{k\to\infty} \ve_k^{-2s} [\bar u_0]_{s,2;\Omega}^2 + \|\bar u_0\|_{2;\Omega}^2 = \ve_0^{-2s} [\bar u_0]_{s,2;\Omega}^2 + \|\bar u_0\|_{2;\Omega}^2 =\frac{S(\Omega_{\ve_0})}{\lambda_{\ve_0}}
$$

These two inequalities combined imply that $w(x) = \bar w(\ve_0^{-1}x)$ is an extremal for $S(\Omega_{\ve_0})$ and so $w=u_0$.

Now, since we are assuming that $d_u F|_{(u_0,\ve_0)}$ is invertible, we can apply the IFT as in the proof of Theorem \ref{existe.mu0} to conclude that for some $\delta>0$ there is a unique extremal for $S(\Omega_\ve)$ for $\ve<\ve_0+\delta$. But this contradicts the definition of $\ve_0$.
\end{proof}

\begin{remark}\label{rem.clave}
By a simple application of the Fredholm's alternative, it follows that $d_u F|_{(u_0,\ve_0)}$ is not invertible if and only if it has a nontrivial kernel.
\end{remark}

The following Poincar\'e-type inequality plays an important role in the bound of $\ve_0$
\begin{lema}\label{Poincare}
Let $\Omega\subset\R^n$ be an open, smooth and of finite measure. Let $0<s<1$ and $1\le q<2^*_s$. Then, there exists $c>0$, that depends on $q,s$ and $\Omega$, such that
$$
c \|w\|_{q;\Omega}^2 \le \frac12 [w]_{s;\Omega}^2,
$$
for every $w\in H^s(\Omega)$ such that $\int_\Omega w\, dx = 0$.
\end{lema}

\begin{proof}
The proof follows by a standard compactness argument and is omitted.
\end{proof}

We are ready to prove the main result of the section.
\begin{teo}\label{cota.mu0}
Under the notations and assumptions of the section, we have that
$$
\ve_0\ge \left(\frac{c}{(q-1) |\Omega|^{1-\frac{2}{q}}}\right)^\frac{1}{2s},
$$
where $c>0$ is the constant in the Poincar\'e-type inequality of Lemma \ref{Poincare}.
\end{teo}

\begin{proof}
By Lemma \ref{lema.clave}, there exists $u_0$ an extremal for $S(\Omega_{\ve_0})$ such that the rescaled function $\bar u_0$ normalized such that $\|\bar u_0\|_{q;\Omega}=1$ verifies that $d_u F|_{(\bar u_0,\ve_0)}$ is not invertible.

Moreover, by Remark \ref{rem.clave}, $d_u F|_{(\bar u_0,\ve_0)}$ has a nontrivial kernel.

Let $0\neq z\in \ker(d_u F|_{(\bar u_0,\ve_0)})$, then $z$ is a nontrivial weak solution to the problem
\begin{equation}\label{eq.z}
\begin{cases}
(-\Delta)^s z + \ve_0^{2s} z = \ve_0^{2s} \lambda_{\ve_0} (q-1) |\bar u_0|^{q-2}z & \text{in }\Omega\\
\int_\Omega z\, dx = 0.
\end{cases}
\end{equation}

Using $z$ as a test function in the weak formulation of \eqref{eq.z} gives
\begin{equation}\label{test.z}
\frac12 [z]_{s,2;\Omega}^2 + \ve_0^{2s} \|z\|_{2;\Omega}^2 = \ve_0^{2s} \lambda_{\ve_0} (q-1)\int_\Omega |\bar u_0|^{q-2} z^2\, dx.
\end{equation}

Now, we use the Poincar\'e-type inequality of Lemma \ref{Poincare}, Lemma \ref{lema1} and H\"older's inequality to deduce from \eqref{test.z} that
$$
c\|z\|_{q;\Omega}^2 \le \ve_0^{2s} |\Omega|^{1-\frac{2}{q}} (q-1) \|\bar u_0\|_{q;\Omega}^{q-2} \|z\|_{q;\Omega}^2.
$$
Finally, recalling that $\|\bar u_0\|_{q;\Omega}=1$ and that $z\neq 0$ we arrive that the desired result.
\end{proof}

\section*{Acknowledgements}
This paper was partially supported by grants UBACyT 20020130100283BA, CONICET PIP 11220150100032CO and ANPCyT PICT 2012-0153. 

J. Fern\'andez Bonder and A. Silva are members of CONICET. 

This work started while JFB was visiting the National University of San Luis. He wants to thank the Math Department and the IMASL for the hospitality that make the visit so enjoyable.

\bibliographystyle{amsplain}
\bibliography{biblio}

\end{document}